\documentclass[12pt,a4paper,twoside]{amsart}

\usepackage[a4paper,top=3cm,bottom=3cm,inner=2.5cm, outer=2.5cm]{geometry}
\usepackage{amsmath,amssymb,amsthm}
\usepackage[english]{babel}
\usepackage[pdftex]{graphicx}
\usepackage[utf8]{inputenc}
\usepackage{bm}
\usepackage{multirow}
\usepackage[shortlabels]{enumitem}
\usepackage{amsfonts}
\usepackage{pst-node}
\usepackage{bussproofs}
\usepackage{placeins}
\usepackage{xspace}
\usepackage{fix-cm}
\usepackage{array}
\usepackage[nomessages]{fp}
\usepackage{microtype,mathabx}
\usepackage[colorlinks,linkcolor=blue,citecolor=violet,urlcolor=gray,final,hyperindex,linktoc=page,hyperfootnotes=true]{hyperref}

\newlength\horspace
\newcommand{\h}[1][1.0]{\hspace{#1\horspace}}
\newlength\verspace
\newlength\negverspace
\usepackage{tikz-cd}
\tikzset{iso/.style={draw=none,every to/.append style={edge node={node [sloped, allow upside down, auto=false]{$\cong$}}}}}
\tikzset{simeq/.style={draw=none,every to/.append style={edge node={node [sloped, allow upside down, auto=false]{$\simeq$}}}}}
\tikzset{simeqS/.style={draw=none,every to/.append style={edge node={node [sloped, allow upside down, auto=false]{$\raisebox{0.8em}{$\simeq$}$}}}}}
\tikzset{aiso/.style={simeqS,preaction={draw,->}}}
\tikzset{simeqSs/.style={draw=none,every to/.append style={edge node={node [sloped, allow upside down, auto=false]{$\raisebox{-0.8em}{\rotatebox{180}{$\simeq$}}$}}}}}
\tikzset{aisos/.style={simeqSs,preaction={draw,->}}}
\tikzset{dotdot/.style={dash pattern=on 0.25ex off 0.2ex, dash phase=0ex}}
\tikzset{RightA/.style={double distance=3.5pt,>={Implies},->},%
	triple/.style={-,preaction={draw,RightA}},%
	quadruple/.style={preaction={draw,RightA,shorten >=0pt},shorten >=1pt,-,double,double distance=0.2pt}}

\newtheorem {teor}{Theorem} [section]

\newtheorem {lemma}[teor]{Lemma}
\newtheorem {prop}[teor]{Proposition}

\newtheorem {costr}[teor]{Construction}

\theoremstyle{definition}
\newtheorem{defne}[teor] {Definition}
\newtheorem{oss}[teor]{Remark}
\newtheorem*{notazione}{Notation}

\newenvironment{cd}{\[\begin{tikzcd}[row sep=7ex,column sep=7ex,ampersand replacement=\&]}{\end{tikzcd}\]\ignorespacesafterend}
\newenvironment{cds}[2]{\[\begin{tikzcd}[row sep=#1ex, column sep=#2ex,ampersand replacement=\&]}{\end{tikzcd}\]\ignorespacesafterend}
\newenvironment{cdN}{\begin{tikzcd}[row sep=7ex,column sep=7ex,ampersand replacement=\&]}{\end{tikzcd}\ignorespacesafterend}
\newenvironment{cdsN}[2]{\begin{tikzcd}[row sep=#1ex, column sep=#2ex,ampersand replacement=\&]}{\end{tikzcd}\ignorespacesafterend}

\newenvironment{eqD*}{\begin{equation*}}{\end{equation*}\ignorespacesafterend}

\def\:{\colon}

\def\phi{\varphi}

\newcommand{\fami}[2]{\left\{{#1}\right\}_{#2}}
\def\dfn#1{{\bfseries\itshape #1}}
\def\predfn#1{{\itshape #1}}
\newcommand{\st}{^{\ast}}

\def\c{\circ}

\DeclareFontFamily{OT1}{pzc}{}
\DeclareFontShape{OT1}{pzc}{m}{it}{<->s*[1.21]pzcmi7t}{}
\DeclareMathAlphabet{\mathpzc}{OT1}{pzc}{m}{it}

\DeclareFontFamily{U}{dutchcal}{\skewchar\font=45}
\DeclareFontShape{U}{dutchcal}{m}{n}{<->s*[1.05] dutchcal-r}{}
\DeclareMathAlphabet{\mathlcal}{U}{dutchcal}{m}{n}

\newcommand{\catfont}[1]{\ensuremath{\mathpzc{#1}}\xspace}
\newcommand{\Cal}[1]{\ensuremath{\mathcal{#1}}\xspace}

\newcommand{\C}{\catfont{C}}
\newcommand{\D}{\catfont{D}}

\newcommand{\Cat}{\catfont{Cat}}
\newcommand{\Gpd}{\catfont{Gpd}}
\newcommand{\Top}{\catfont{Top}}

\makeatletter

\newcommand{\Sh}[2][\@nil]{%
	\def\tmp{#1}%
	\ifx\tmp\@nnil{\ensuremath{\catfont{Sh}\hspace{-0.15ex}\left({#2}\right)}}%
	\else{\ensuremath{\catfont{Sh}\hspace{-0.15ex}\left({#2},{#1}\right)}}\fi}

\newcommand{\St}[2][\@nil]{%
	\def\tmp{#1}%
	\ifx\tmp\@nnil{\ensuremath{\catfont{St}\hspace{-0.15ex}\left({#2}\right)}}%
	\else{\ensuremath{\catfont{St}\hspace{-0.15ex}\left({#2},{#1}\right)}}\fi}
\makeatother

\newcommand{\x}[1][]{\h[-1]\times_{#1}\h[-1]}
\newcommand{\opn}[1]{\operatorname{#1}}
\newcommand{\id}[1]{\operatorname{id}_{#1}}

\newcommand{\op}{\ensuremath{^{\operatorname{op}}}}

\newcommand{\restr}[2]{{\left.\kern-\nulldelimiterspace {#1}\vphantom{\big|} \right|_{#2}}}
\newcommand{\dom}{\operatorname{dom}}
\newcommand{\pr}[1]{\operatorname{pr}_{#1}}

\DeclareMathOperator{\colim}{colim}

\makeatletter
\newcommand{\ar}[2][]{\xrightarrow[#1]{#2}}
\newcommand{\aar}[2][]{\xrightarrow[#1]{#2}} 
\def\xlongrightarrowfill@{\arrowfill@\relbar\relbar\longrightarrow}
\newcommand{\arr}[2][]{%
	\ext@arrow 0099\xlongrightarrowfill@{#1}{#2}}
\newcommand{\aarr}[2][]{%
	\ext@arrow 0099\xlongrightarrowfill@{#1}{#2}} 

\newcommand{\aR}[2][]{%
	\ext@arrow 0055{\Rightarrowfill@}{#1}{#2}}
\def\xLongrightarrowfill@{\arrowfill@\Relbar\Relbar\Longrightarrow}
\newcommand{\aRR}[2][]{%
	\ext@arrow 0099\xLongrightarrowfill@{#1}{#2}}
\def\aitofill@{\arrowfill@{\lhook\joinrel\relbar}\relbar\rightarrow}
\newcommand{\aito}[2][]{%
	\ext@arrow 3095\aitofill@{#1}{#2}}
\def\Longaitofill@{\arrowfill@{\lhook\joinrel\relbar\joinrel\relbar}\relbar\rightarrow}
\newcommand{\aitoo}[2][]{%
	\ext@arrow 0099\Longaitofill@{#1}{#2}}

\def\xlongleftarrowfill@{\arrowfill@\longleftarrow\relbar\relbar}
\newcommand{\all}[2][]{%
	\ext@arrow 0099\xlongleftarrowfill@{#1}{#2}}
\newcommand{\aL}[2][]{%
	\ext@arrow 0055{\Leftarrowfill@}{#1}{#2}}
\def\xLongleftarrowfill@{\arrowfill@\Longleftarrow\Relbar\Relbar}
\newcommand{\aLL}[2][]{%
	\ext@arrow 0099\xLongleftarrowfill@{#1}{#2}}

\def\xmapstofill@{\arrowfill@{\mapstochar\relbar}\relbar\rightarrow}
\newcommand{\am}[2][]{%
	\ext@arrow 0395\xmapstofill@{#1}{#2}}
\def\xlongmapstofill@{\arrowfill@\relbar\relbar\longmapsto}
\newcommand{\amm}[2][]{%
	\ext@arrow 0399\xlongmapstofill@{#1}{#2}}

\newcommand{\eqq}{\DOTSB\protect\Relbar\protect\joinrel\Relbar}
\def\xeqqfill@{\arrowfill@\Relbar\Relbar\eqq}
\newcommand{\aeqq}[2][]{%
	\ext@arrow 0099\xeqqfill@{#1}{#2}}

\def\xRrightarrowfill@{\arrowfill@\equiv\equiv\Rrightarrow}
\newcommand{\aM}[2][]{\ext@arrow 0359\xRrightarrowfill@{#1}{#2}}
\newcommand{\Llongrightarrow}{%
	\DOTSB\protect\equiv\protect\joinrel\Rrightarrow}
\def\xLlongrightarrowfill@{\arrowfill@\equiv\equiv\Llongrightarrow}
\newcommand{\aMM}[2][]{%
	\ext@arrow 0099\xLlongrightarrowfill@{#1}{#2}}
\makeatother

\newcommand{\aiso}{\aar{\scriptstyle\simeq}}
\newcommand{\aisoo}{\aarr{\scriptstyle\simeq}}
\newcommand{\aequi}{\ensuremath{\stackrel{\raisebox{-1ex}{\kern-.3ex$\scriptstyle\sim$}}{\rightarrow}}}
\newcommand{\aequii}{\ensuremath{\stackrel{\raisebox{-1ex}{\kern-.3ex$\scriptstyle\sim$}}{\longrightarrow}}}

\newcommand{\PB}[1]{\arrow[#1,phantom,"\scalebox{1.6}{\color{black}$\lrcorner$}",very near start]}

\newcommand{\sq}[8]{%
	\begin{cd}
		{#1}\arrow[r,"{#5}"]\arrow[d,"{#6}"']\&{#2}\arrow[d,"{#7}"]\\
		{#3}\arrow[r,"{#8}"']\&{#4}
\end{cd}}
\newcommand{\sqN}[8]{%
	\begin{cdN}
		{#1}\arrow[r,"{#5}"]\arrow[d,"{#6}"']\&{#2}\arrow[d,"{#7}"]\\
		{#3}\arrow[r,"{#8}"']\&{#4}
\end{cdN}}

\newcommand{\tr}[7][4.5]{%
	\begin{cds}{6.5}{#1}
		{#2}\arrow[rr,"{#5}"]\arrow[dr,"{#6}"']\&\&{#3}\arrow[ld,"{#7}"]\\
		\&{#4}
\end{cds}}

\newcommand{\pbsq}[8]{%
	\begin{cd}
		#1 \arrow[r,"{#5}"] \arrow[d,"{#6}"'] \PB{rd} \& #2 \arrow[d,"{#7}"] \\
		#3 \arrow[r,"{#8}"'] \& #4
\end{cd}}

\newcommand{\pbsqunivv}[9][]{%
	\def\foopbsqunivv##1##2##3##4{%
		\begin{cd}
			#2\arrow[rrd,bend left,"{#3}"]\arrow[rdd,bend right=35,"{#4}"']\arrow[rd,dashed,"{#5}"{#1}]\&[-4ex]\\[-4ex]
			\&#6 \arrow[r,"{##1}"] \arrow[d,"{##2}"'] \PB{rd} \& #7 \arrow[d,"{##3}"] \\
			\&#8 \arrow[r,"{##4}"'] \& #9
	\end{cd}}%
	\foopbsqunivv%
}
\newcommand{\pbsqunivvN}[9][]{%
	\def\foopbsqunivvN##1##2##3##4{%
		\begin{cdN}
			#2\arrow[rrd,bend left,"{#3}"]\arrow[rdd,bend right=35,"{#4}"']\arrow[rd,dashed,"{#5}"{#1}]\&[-4ex]\\[-4ex]
			\&#6 \arrow[r,"{##1}"] \arrow[d,"{##2}"'] \PB{rd} \& #7 \arrow[d,"{##3}"] \\
			\&#8 \arrow[r,"{##4}"'] \& #9
	\end{cdN}}%
	\foopbsqunivvN%
}

\def\dfn#1{\textbf{#1}}
\def\predfn#1{\textit{#1}}

\newcommand{\cF}{\Cal{F}}

\newcommand{\Bun}[2]{\catfont{Bun}_{#1}(#2)}
\newcommand{\qst}[2]{[{#1}/{#2}]}
\newcommand{\coequalizer}[7][0.8]{%
	\begin{cd}
		{#2}\arrow[r,shift left=#1ex,"{#5}"]\arrow[r,shift right=#1ex,"{#6}"']\& {#3} \arrow[r,"{#7}"]\&{#4}
\end{cd}}
\newcommand{\shd}[1]{^{\hphantom{*}}_{#1}}
\newcommand{\clst}[1]{\Cal{B}#1}

\begin{document}
	\title[Generalized quotient stacks]{Generalized principal bundles and quotient stacks}
	\author[E. Caviglia]{Elena Caviglia}
	\address{School of Computing and Mathematical Sciences, University of Leicester}
	\email{ec363@leicester.ac.uk}

	\begin{abstract}
		We consider the internalization of the usual notion of principal bundle in a site that has all pullbacks and a terminal object. We use this notion to consider the explicit construction of quotient prestacks via presheaves of categories of principal bundles equipped with equivariant morphisms in this abstract context. We then prove that, if the site is subcanonical and the underlying category satisfies some mild conditions, these quotient prestacks satisfy descent in the sense of stacks.
	\end{abstract}

\date{\today}
\maketitle
\setcounter{tocdepth}{1}
\tableofcontents

\section{Introduction}

 Principal bundles over topological spaces and smooth manifolds have been originally studied in geometry and topology (see \cite{Milnor}, \cite{Husemoller} or \cite{Steenrod}). Over the years the notion of principal bundle has been generalized in different directions. Kock in \cite{Kock} studied generalized fibre bundles introducing a notion that makes sense in any category with finite inverse limits, while a notion of principal bundle in a Grothendieck topos has been studied, even in the infinite dimensional case, by Nikolaus, Schreiber  and Stevenson in \cite{NikolausSchreiberStevenson}. 

In this paper we consider an internalized notion of  principal bundle that makes sense in any site, provided that the underlying category has all pullbacks and a terminal object. The topological group involved in the standard notion becomes here a group object of the category and the notion of locally trivial morphism is internalized by considering pullbacks along the morphisms of a covering family for the Grothendieck topology of the site. This internalized notion was essentially present  in Grothendieck's work (see \cite{GrTechcostr} and\cite{GrTechDesc}) and a similar notion has been used also by Sati and Schreiber in \cite{SatiSchreiber}.
This internal notion coincides with the classical one when the site is $(\Top, std)$, where $\Top$ is the category of compactly generated Hausdorff topological spaces (and continuous maps between them) and $std$ is the standard Grothendieck topology, i.e.\ the covering families of a topological space coincide with its open coverings. Moreover, also principal bundles in the algebraic geometry context (see \cite{LaumonMoret} or \cite{Neumann}) are an instance of this notion.

We believe that one of the main advantages of this internal notion is that it could make possible to use the rich and important theory of principal bundles in new contexts of geometry and mathematical physics.

The importance of principal bundles lies also in the fact that they are the objects classified by classifying spaces via homotopy pullbacks (see \cite{Milnor}) and by classifying stacks via 2-pullbacks. Classifying stacks are particular cases of the very important notion of quotient stacks. 
Quotient stacks are commonly thought as stackifications of presheaves of action groupoids (see as a classical reference the book \cite{LaumonMoret} by Laumon and Moret-Bailly), but an explicit construction that uses presheaves of groupoids of principal bundles equipped with equivariant maps to a fixed space is also present in the literature in the algebraic case and in the differentiable case (see, for instance, \cite{Neumann} and \cite{Heinloth} respectively). In this paper we perform the same construction in the internal context to define quotient prestacks. We, then, prove that these objects are well defined pseudofunctors. A crucial difference between these quotient prestacks and the classical ones is that in our case  prestacks take values in $\Cat$ and not necessarily in $\Gpd$, while classical quotient stacks are stacks of groupoids. This is due to the fact that a generic Grothendieck topology does not always behave well with respect to the morphisms of the category in a sense that is explained in Remark~\ref{notGpd}.

Since the classical quotient prestacks are always stacks, it is natural to wonder whether this is true for the internal ones. The answer to this question seems negative, but a proof of this is not given in this paper. However, we find sufficient abstract conditions that guarantee that  these generalized quotient prestacks are satisfy descent in the sense of stacks. This is the content of the main result of this paper:
\begin{teor}\label{maint}
	Let $\C$ be a  cocomplete category with pullbacks and a terminal object and such that pullbacks preserve colimits. Let then $\tau$ be a subcanonical Grothendieck topology on $\C$ and $X,G\in \C$  with $G$ a group object. Then the quotient prestack $[X/G]$ is a stack.
\end{teor}
Notice that the required conditions on the category $\C$ are not so restrictive. They are, in fact, satisfied by all locally cartesian closed categories with a terminal object. This gives us many contexts in which it is possible to apply our result, such as all cocomplete quasitopoi (see \cite{Penon}) that include the interesting case of diffeological spaces in the sense of \cite{Souriau} as well as all cocomplete elementary topoi.
To prove this theorem, that appears in Section~\ref{st} as Theorem~\ref{arestacks}, we show that the quotient prestack is a stack when the considered topology is the \predfn{canonical topology}. To do so, we use the basis of the canonical topology introduced by Lester in \cite{Lester1}.

The proof  of Theorem \ref{maint} is rather technical, especially when it comes to prove that every descent datum is effective, and we need to use the language of sieves, instead of  the one of covering families, to avoid even more complicated technicalities. We recall this useful version of the definition of descent datum that is surely well-known, but for which we could not find a reference. 
\begin{defne}
	Let $\cF\: \C\op \to \Cat$ be a prestack and let $S$ be a sieve on $Y\in \C$. A \dfn{descent datum on $S$ for $\cF$} is an assignment for every morphism $Z\ar{f} Y$ in $S$ of an object $W_f \in \cF(Z)$ and, for every pair of composable morphisms 
	$Z'\ar{g}Z\ar{f}Y$
	with $f\in S$, of an isomorphism 
	${\phi^{f,g}} \: g\st W_f \aisoo W_{g\c f}$ such that, given morphisms
	$Z''\ar{h} Z' \ar{g}  Z\ar{f} Y$
	with $f\in S$, the following diagram is commutative
	\begin{cd}
		h\st (g\st W_f )\arrow[d,"", aiso] \arrow[r,"h\st \phi^{f,g}"] \&   h\st (W_{f\c g}) \arrow[d,"\phi^{f\c g,h}"] \\
		(g\c h)\st (W_f) \arrow[r,"\phi^{f,g\c h}"']\& {W_{f\c g\c h}.}
	\end{cd}
		This descent datum is called \dfn{effective} if there exist an object $W\in \cF(Y)$ and, for every morphism $Z\ar{f}Y\in S$, an isomorphism 
	$$\psi^{f}\: f\st(W) \aisoo W_f$$
	such that, given morphisms
	$Z'\ar{g} Z \ar{f} Y$
	with $f\in S$, the following diagram is commutative
	\begin{cd}
		g\st (f\st (W)) \arrow[d,"",aiso] \arrow[r,"g\st \psi^{f}"] \&   {g\st(W_f)} \arrow[d,"\phi^{f,g}"] \\
		(f\c g)\st W\arrow[r,"\psi^{f\c g}"]\& {W_{f\c g}.}
	\end{cd}
\end{defne}
A stack for us will be a pseudofunctor that satisfies the descent condition up to equivalence (see, for instance, \cite{Neumann}), as in the following definition:
\begin{defne}
	A prestack $\cF\: \C\op \to \Cat$ is a \dfn{stack} if it satisfies the following conditions:
	\begin{itemize}
		\item[-] Every descent datum for $\cF$ is effective;
		\vspace{2.5mm}
		\item[-](\predfn{Gluing of morphisms}) Given a covering family $\mathcal{U}=\fami{f_i\colon U_i \to U}{i\in I}$, objects $x$ and $y$ of $\cF(U)$ and morphisms $\phi_i\colon \restr{x}{U_i} \to \restr{y}{U_i}$ in $\cF(U_i)$ for every $i\in I$ such that for every $i,j\in I$
		$\restr{\phi_i}{U_{ij}}=\restr{\phi_j}{U_{ij}},$
		there exists a morphism $\eta\colon x \to y$ such that $\restr{\eta}{U_i}=\phi_i$;
		\vspace{2.5mm}
		\item[-] (\predfn{Uniqueness of gluings}) Given a covering family $\mathcal{U}=\fami{f_i\colon U_i \to U}{i\in I}$, objects $x$ and $y$ of $\cF(U)$ and morphisms $\phi,\psi:x\to y$ such that for every $i\in I$
		$\restr{\phi}{U_i}=\restr{\psi}{U_i},$
		then $\phi=\psi$.
	\end{itemize}
\end{defne}
It is well-known that in the classical case isomorphism classes of principal $G$-bundles are in bijection with the first cohomology group with coefficients in $G$ (see \cite{Serre} for a proof of this fact) and we believe that an analogous, but stronger, result can be proved in our generalized setting, taking into account not only the objects but the whole categorical structure on both sides. This result, together with a 2-dimensional version of the theory presented in this paper, is work in progress of the author.

\subsection*{Outline of the paper}
 In section~\ref{prinbun}, after recalling the notions of action and equivariant morphism in the context of group objects of a category, we construct an action of an internal group on a pullback, given its actions on the sources of the two morphisms involved in the pullback (Construction~\ref{pullact}). This particular action is needed to internalize the notion of principal bundle and will be used several times throughout the paper. Moreover, we introduce the notion of \predfn{locally trivial morphism} in a site (Definition~\ref{loctriv}) and we use it  to define \predfn{principal bundles} and morphisms between them in this context  (Definitions~\ref{defprinbun} and~\ref{morprinbun}). We conclude the section observing that this notion truly generalizes the classical notion of principal bundle over a topological space. Section~\ref{pr} is dedicated to the definition of generalized \predfn{quotient prestacks} (Definition~\ref{qp}). A key result of the section is that generalized principal bundles are stable under pullbacks (Proposition~\ref{pullbun}), as this is crucial to prove that our definition of quotient prestack is a good one  (Proposition~\ref{welldef}).
 We start the last section of the paper (Section~\ref{st}) recalling some useful results about the \predfn{canonical topology}. We, then, prove the main result of the paper (Theorem~\ref{arestacks}) which states that, under certain mild assumptions on the underneath category, the quotient prestacks are stacks with respect to every subcanonical topology.

\section{Generalized principal bundles} \label{prinbun}
Let $\C$ be a category with pullbacks and terminal object $T$ and let $\tau$ be a Grothendieck topology on it. In this section we will consider \predfn{principal $G$-bundles} over $X$, where $X$ is an object of $\C$ and $G$ is an internal group in $\C$.

Firstly, we recall the well-known definitions of \predfn{action} of an internal group of $\C$ on another object of the category and \predfn{equivariant} morphism (see Grothendieck's \cite{GrTechcostr} and Heckmann and Hilton's \cite{EichmannHilton} and \cite{EckmannHilton2} as classical references).

\begin{defne}
	Let $G$ be an internal group in $\C$ and let $X$ be an object of $\C$. An \dfn{action of $G$ on $X$} is a morphism
	$$x\: G\x X \to X$$
	such that the following diagrams are commutative
	\begin{eqD*}
	\sqN{G\x G\x X}{G\x X}{G\x X}{X}{\id{G}\x x}{m\x \id{X}}{x}{x}
	\quad \quad
	\sqN{T\x X}{G\x X}{X}{X,}{e\x \id{X}}{\pr{2}}{x}{\id{X}}
\end{eqD*}
where $T$ is the terminal object of $\C$, $m\: G \x G \to G$ is the internal multiplication of $G$ and $e: T \to G$ is the internal neutral element of $G$.
\end{defne}

\begin{defne}
	Let $G$ be a group object in $\C$ that acts on  the objects $X$ and $Y$ of $\C$ with actions $x\: G\x X \to X$ and $y\: G\x Y \to Y$ respectively. A  \dfn{$G$-equivariant morphism} $f\: X \to Y$ is a morphism in $\C$ such that the following square is commutative:
	\sq{G\x X}{X}{G\x Y}{Y.}{x}{\id{G}\x f}{f}{y} 
\end{defne}

\begin{oss}
	Notice that every morphism of $\C$ is $G$-equivariant when the source and the target are equipped with trivial actions of $G$. Having this in mind, we will always think of the objects of $\C$ as equipped with trivial actions of $G$ when the action is not specified.
\end{oss}

We now describe how to define an action of $G$ on the pullback of two morphisms in $\C$, given actions of $G$ on the domains of the the morphisms. This action will be largely used throughout the article.

\begin{costr} \label{pullact}
	Let $G$ be a group object of $\C$ that acts on $P\in \C$ with action $p\: G\x P\to P$, on $Y\in \C$ with action $y\: G\x Y \to Y$ and on $Z\in\C$ with action $z\: G\x Z \to Z$ and let $f:P\to Y$ and $g\: Z\to Y$ be $G$-equivariant morphisms. 
	
	Consider, then, the pullback square
	\pbsq{P\x[Y]Z}{P}{Z}{Y.}{f\st g}{g\st f}{f}{g} 
	We want to define an action of $G$ on the pullback $P\x[Y]Z$ and we can do this using the morphism $\psi\: G\x (P\x[Y]Z) \to P\x[Y]Z$ induced by the universal property of the pullback $P\x[Y]Z$ as in the following diagram
	\pbsqunivv{G\x(P\x[Y]Z)}{p\c(\id{G}\x f\st g)}{z \hspace{0.1ex}\c (\id{G}\x g\st f)}{\psi}{P\x[Y]Z}{P}{Z}{Y,}{f\st g}{g\st f}{f}{g} 
	where the outer square is commutative since it can be obtained by gluing commutative squares as follows
	\begin{cd}
		G\x(P\x[Y]Z) \arrow[r,"\id{G}\x f\st g"]  \arrow[d,"\id{G}\x g\st f"] \& G\x P \arrow[r,"p"] \arrow[d,"\id{G}\x f"] \& P \arrow[d,"f"]\\
		G\x Z \arrow[r,"\id{G} \x g"] \arrow[d,"z"] \& G\x Y \arrow[r,"y"] \& Y \arrow[d,equal,""]\\
		Z \arrow[rr,"g"] \& \& Y,
		\end{cd}
	where the upper right square is commutative, because $f$ is equivariant and the lower square because $g$ is equivariant.
\end{costr}

\begin{prop}
	The morphism $\psi\: G\x (P\x[Y]Z) \to P\x[Y]Z $ defined in Construction~\ref{pullact} is an action of $G$ on $P\x[Y]Z$.
\end{prop}

\begin{proof}
	To prove that $\psi$ is an action we need to show that the following diagrams commute:
		\begin{eqD*}
		\sqN{G\x G\x (P\x[Y]Z)}{G\x (P\x[Y]Z)}{G\x (P\x[Y]Z)}{P\x[Y]Z}{\id{G}\x \psi}{m\x \id{P\x[Y]Z}}{\psi}{\psi}
		\quad \quad
		\sqN{T\x (P\x[Y]Z)}{G\x( P\x[Y]Z)}{P\x[Y]Z}{P\x[Y]Z.}{e\x \id{P\x[Y]Z}}{\pr{2}}{\psi}{\id{P\x[Y]Z}}
	\end{eqD*}
To do this, we will use the universal property of the pullback $P\x[Y]Z$.

We start from the diagram on the left and we consider the post-composition with $f\st g$, so we need to prove that $f\st g \c \psi\c (\id{G}\x \psi)=f\st g \c \psi\c (m\x \id{P\x[Y]Z})$. This is shown by the following commutative diagram
\begin{cd}
	{G\x G\x (P\x[Y]Z)} \arrow[rr,"\id{G}\x \psi"] \arrow[ddd,"m\x \id{P\x[Y]Z}"] \arrow[rd,"\id{G}\x(\id{G}\x f\st g)"]\& \& {G\x (P\x[Y]Z)} \arrow[r,"\psi"] \arrow[d,"\id{G}\x f\st g"]\& {P\x[Y]Z} \arrow[ddd,"f\st g"] \\
	\& {G\x G\x P} \arrow[r,"\id{G}\x p"] \arrow[d,"m\x \id{P}"]\& {G\x P} \arrow[rdd,"p"] \&  \\
	\& {G\x P}\arrow[rrd,"p"]\& \& \\
	{G\x (P\x[Y]Z)} \arrow[ru,"\id{G}\x f\st g"] \arrow[r,"\psi"] \&{P\x[Y]Z} \arrow[rr,"f\st g"] \& \& {P.}  
	\end{cd}
Analogously, we have to prove that $g \st f\c \psi\c (\id{G}\x \psi)=g \st f \c \psi\c (m\x \id{P\x[Y]Z})$ and this is shown by the commutative diagram
\begin{cd}
	{G\x G\x (P\x[Y]Z)} \arrow[rr,"\id{G}\x \psi"] \arrow[ddd,"m\x \id{P\x[Y]Z}"] \arrow[rd,"\id{G}\x(\id{G}\x g \st f)"]\& \& {G\x (P\x[Y]Z)} \arrow[r,"\psi"] \arrow[d,"\id{G}\x g \st f"]\& {P\x[Y]Z} \arrow[ddd,"g \st f"] \\
	\& {G\x G\x P} \arrow[r,"\id{G}\x z"] \arrow[d,"m\x \id{Z}"]\& {G\x Z} \arrow[rdd,"z"] \&  \\
	\& {G\x Z}\arrow[rrd,"z"]\& \& \\
	{G\x (P\x[Y]Z)} \arrow[ru,"\id{G}\x g \st f"] \arrow[r,"\psi"] \&{P\x[Y]Z} \arrow[rr,"g \st f"] \& \& {Z.}  
\end{cd}
By the universal property of the pullback $P\x[Y]Z$, we conclude that the first property required to be an action holds.
To prove that the second property holds as well, we use again the universal property of the pullback $P\x[Y]Z$. 

\noindent The following diagram shows that $f\st g \c \psi \c  e \x \id{P\x[Y]Z}=f\st g \c \pr{2}$:
\begin{cd}
	{T \x (P\x[Y]Z)}  \arrow[dd,"\pr{2}"']\arrow[rr,"e\x \id{P\x[Y]Z}"] \arrow[rd,"\id{T}\x f\st g"] \& \& {G\x (P\x[Y]Z)} \arrow[r,"\psi"] \arrow[d,"\id{G}\x f\st g"] \& {P\x[Y]Z}\arrow[dd,"f\st g"]\\
	\& {T\x P} \arrow[r,"e\x \id{P}"] \arrow[rrd,"\pr{2}"', end anchor= {[yshift=0.5ex]}]\&  {G\x P} \arrow[rd,"p", end anchor= {[yshift=1ex]}]\& \\
	{P\x[Y]Z} \arrow[rrr,"f\st g"']\& \&  \& {P.}
\end{cd}
Analogously, the following diagram shows that $g\st f \c \psi \c  e \x \id{P\x[Y]Z}=g\st f \c \pr{2}$:
\begin{cd}
	{T \x (P\x[Y]Z)}  \arrow[dd,"\pr{2}"']\arrow[rr,"e\x \id{P\x[Y]Z}"] \arrow[rd,"\id{T}\x g \st f"] \& \& {G\x (P\x[Y]Z)} \arrow[r,"\psi"] \arrow[d,"\id{G}\x g \st f"] \& {P\x[Y]Z}\arrow[dd,"g \st f"]\\
	\& {T\x Z} \arrow[r,"e\x \id{Z}"] \arrow[rrd,"\pr{2}"', end anchor= {[yshift=0.5ex]}]\&  {G\x Z} \arrow[rd,"z", end anchor= {[yshift=1ex]}]\& \\
	{P\x[Y]Z} \arrow[rrr,"g \st f"']\& \&  \& {Z.}
\end{cd}
Hence we conclude that $\psi \c  e \x \id{P\x[Y]Z}=\pr{2}$ and so the second property of action holds for $\psi$.
	\end{proof}

\begin{oss}
	Notice that the result of Proposition \ref{pullact} can be seen as consequence of the fact that $G$-actions are the algebras for the monad $G\x -$ and monadic functors create all limits that exist. This is the approach used in the case of topological spaces by Sati and Schreiber in \cite{SatiSchreiber}.
\end{oss}

In order to introduce the concept of principal bundle in this more general context, we need to define \predfn{locally trivial} morphisms with respect to the fixed Grothendieck topology on $\C$. 

\begin{defne} \label{loctriv}
		Let $g\: Y\to X$ be a morphism of $\C$. We say that $g$ is \dfn{locally trivial} if there exists a covering $\fami{f_i\: U_i\to X}{i\in I}$ of $X$ such that for every $i\in I$ the pullback 
	\pbsq{Y\x [X]U_i}{Y}{U_i}{X}{g \st f_i}{f_i\st g}{g}{f_i}
	is isomorphic to $G\x U_i$ and the isomorphism 
	$$\phi_i\:  Y\x[X] U_i \to G\x U_i$$
	is $G$-equivariant, i.e. 
	\sq{G\x(Y\x[X]U_i)}{Y\x[X]U_i}{G\x (G\x U_i)}{G\x U_i,}{\psi_i}{\id{G}\x \phi_i}{\phi_i}{\theta_i}
	where the action $\theta_i\: G\x (G\x U_i) \to G\x U_i$ is the composition of the canonical isomorphism $G\x (G\x U_i)\aiso G\x U_i$ with the internal multiplication $m: G\x G \to G$ of $G$ and the action $\psi_i\: G\x (Y\x[X] U_i) \to  Y\x[X] U_i$ is defined as in Construction~\ref{pullact}.
\end{defne}

We are now ready to give the main definition of this section.

\begin{defne} \label{defprinbun}
	Let $G$ be an internal group in the site $(\C, \tau)$ and let $X$ be an object of $\C$. A \dfn{principal $G$-bundle} over $X$ is a a $G$-equivariant locally trivial morphism $\pi_P\: P\to X$, where the object $P$ of $\C$ is  equipped by an action $p\: G\x P\to P$.
\end{defne}
 
 \begin{defne} \label{morprinbun}
 	Let $\pi_P\: P\to X$ and $\pi_Q\: Q\to X$ be principal $G$-bundles over $X$ in $\C$. A \dfn{morphism of principal $G$-bundles} $\phi$ from $\pi_P\: P\to X$ to $\pi_Q\: Q\to X$ is a morphism $\phi\: P \to Q$ in $\C$ such that 
 	\tr{P}{Q}{X.}{\phi}{\pi_P}{\pi_Q}
 \end{defne}

\begin{notazione}
	Principal $G$-bundles over $X$ and morphisms of principal bundles between them form a category that we will denote $\Bun{G}{X}$.
\end{notazione}

\begin{oss}
	 Consider the category $\Top$ is the category of compactly generated Hausdorff topological spaces and $std$ is the standard topology on it, i.e.\ the coverings of a space are the usual open coverings (with inclusions of the opens in the space as morphisms). In this setting our notion of principal bundles coincides with the usual one.
	 
	 Notice that up to now we do not need to restrict ourselves to only consider compactly generated Hausdorff spaces but we have chosen to do so both because this category is much more well-behaved than the category of all topological spaces and because this way we can apply Theorem \ref{arestacks} to $(\Top, std)$ as pullbacks preserve colimits in this category.
\end{oss}

\section{Generalized quotient prestacks} \label{pr}
In this section we will explicitly construct quotient prestacks over a site using the traditional construction in our internal context. We will then prove that they are well-defined pseudofunctors. To do so, we will also show that the principal bundles introduced in Definition \ref{defprinbun} are stable under pullbacks.

We are now ready to give the explicit construction of quotient prestacks in our context. This construction follows a well-known recipe that is present in the literature in the case of schemes over a field (see for instance page 28 of \cite{Neumann}) and in the differentiable case (see for instance Example 1.5 of \cite{Heinloth}).

\begin{defne} \label{qp}
	Let $\C$ be a category with pullbacks and terminal object and let $\tau$ be a Grothendieck topology on it. Let then $X$ be an object of $\C$ and let $G$ be a group object of $\C$  that acts on $X$ with action $x\: G\x X \to X$ 
	The \dfn{quotient prestack} 
	$$\qst{X}{G}\: \C\op \to \Cat$$
	is defined as follows:
	\begin{itemize}
		\item for every object $Y\in \C$ we define $\qst{X}{G}(Y)$ as the category that has
		\begin{itemize}
			\item as objects the pairs $(P,\alpha)$ where $\pi_P\: P\to Y$ is a principal $G$-bundle over $Y$ and $\alpha\: P\to X$ is a $G$-equivariant morphism;
			\item as morphisms from the object $(P,\alpha)$ to the object $(Q,\beta)$ the morphisms of principal $G$-bundles $\phi\: P\to Q$ such that
			\tr{P}{Q}{X;}{\phi}{\alpha}{\beta}
		\end{itemize}
	\item for every morphism $f\: Z\to Y$ in $\C$, the functor $$\qst{X}{G}(f)\: \qst{X}{G}(Y) \to \qst{X}{G}(Z)$$ sends
	\begin{itemize}
		\item an object $(P,\alpha)\in \qst{X}{G}(Y)$ to the pair $(P\x[Y]Z, \alpha \c \pi_P \st f)$, where $P\x[Y]Z$ is the pullback
		\pbsq{P\x[Y]Z}{P}{Z}{Y;}{\pi_P \st f}{f \st \pi_P}{\pi_P}{f}
		\item a morphism $\phi\: (P,\alpha) \to (Q,\beta)$ to the morphism $\qst{X}{G}(f)(\phi)$ defined by the universal property of the pullback as in the following diagram
		\pbsqunivv{P\x[Y]Z}{\phi\c \pi_P\st f}{\pi_P}{\qst{X}{G}(f)(\phi)}{Q\x[Y]Z}{Q}{Z}{Y,}{\pi_Q \st f}{f\st \pi_Q}{\pi_Q}{f}
		where the biggest square is commutative because $\phi$ is $G$-equivariant.
	\end{itemize}
	\end{itemize}
\end{defne}

\begin{oss} \label{notGpd}
	While the classical quotient prestacks take values in $\Gpd$, this is not necessarily true for the ones just introduced. This is due to the fact that in a generic site it is not always possible to define morphisms locally (on a covering) and then use the local definitions to define a global morphism. The classical proofs of the fact that morphisms of principal bundles over topological spaces are always isomorphisms (see for instance \cite{Mitchell}, Proposition 2.1) strongly rely on this special property of the site $(\Top, std)$. 
\end{oss}

In order to show that $\qst{X}{G}$ is well-defined and that it is a prestack (i.e. a pseudofunctor that takes values in $\Cat$), we will need the following result.

\begin{prop} \label{pullbun}
	Let $\pi_P\: P\to Y$ be a principal $G$-bundle  over $Y$. Then the pullback $P\x[Y]Z$ given by the following pullback square 
	\pbsq{P\x[Y]Z}{P}{Z}{Y}{\pi_P\st f}{f\st \pi_P}{\pi_P}{f}
	is a principal $G$-bundle over $Z$ with morphism $\pi_{P\x[Y]Z}:=f\st \pi_P$.
\end{prop}

\begin{proof}
	
	We begin considering the action $\psi\: G\x (P\x[Y]Z) \to P\x[Y]Z$ of $G$ on the pullback $P\x[Y]Z$ defined as in Construction~\ref{pullact}.
	
	We notice that the morphism $\pi_{P\x[Y]Z}$ is $G$-equivariant since the square 
	\sq{G\x(P\x[Y]Z)}{P\x[Y]Z}{G\x P}{P,}{\psi}{\id{G}\x f\st \pi_P}{f\st \pi_P}{p}
  is commutative by definition of $\psi$.
	
	We now need to show that $\pi_{P\x[Y]Z}$ is locally trivial.
	
	\noindent Since $\pi_P\: P \to Y$ is a principal $G$-bundle, there exists a covering $\mathcal{V}=\fami{g_i\: V_i\to Y}{i\in I}$ of $Y$ such that for every $i\in I$ the pullback 
	\pbsq{P\x[Y]V_i}{P}{V_i}{Y}{\pi_P\st g_i}{g_i\st \pi_P}{\pi_P}{g_i}
	is isomorphic to $G\x V_i$ and the isomorphism $\alpha_i\: P\x[Y]V_i\to G\x V_i$ is $G$-equivariant (here the action of $G$ on $P\x[Y]V_i$ is defined as in Construction~\ref{pullact}).
	Therefore, we can use the covering $\mathcal{V}$ of $Y$ to define the family $\mathcal{U}=\fami{f\st g_i\: V_i\x[Y]Z\to Z}{i\in I}$ that is a covering of $Z$ by definition of Grothendieck topology.
	In order to use this covering to prove that $\pi_{P\x[Y]Z}$ is locally trivial, we need to show that for every $i\in I$ there exists an isomorphism
	$$\phi_i\: (P\x[Y]Z) \x[Z] (V_i\x[Y] Z) \to G\x (V_i\x[Y] Z).$$
	Equivalently, we need to show that $G\x (V_i\x[Y] Z)$ satisfies the universal property of the pullback $(P\x[Y]Z) \x[Z] (V_i\x[Y] Z)$. To do this, we define a morphism 
	$$h_i\: G\x(V_i\x[Y] Z) \to P\x[Y]Z$$
	for every $i\in I$ as shown in the following diagram
	\pbsqunivv{G\x(V_i\x[Y]Z)}{\pi_P\st g_i \c \hspace{0.3ex}\alpha_i^{-1}\c \hspace{0.3ex} (\id{G}\x g_i\st f)}{\pr{2}\c (\id{G} \x f\st g_i)}{h_i}{P\x[Y] Z}{P}{Z}{Y}{\pi_P\st f}{f\st\pi_{P}}{ \pi_P}{f}
and we want to show that $G\x (V_i\x[Y] Z)$ with morphisms $h_i\: G\x(V_i\x[Y] Z) \to P\x[Y]Z$ and $\pr{2}\: G\x (V_i\x[Y] Z) \to V_i\x[Y] Z$ satisfies the universal property of the pullback $(P\x[Y]Z) \x[Z] (V_i\x[Y] Z)$.
Firstly, we need to prove that the following diagram is commutative:
\sq{G\x(V_i\x[Y] Z)}{P\x[Y] Z}{V_i\x[Y]Z}{Z.}{h_i}{\pr{2}}{f\st \pi_P}{f\st g_i}
This is true since we can obtain this square by gluing a commutative square given by the definition of $h_i$ and a trivially commutative square.

Let now $C$ be an object of $\C$ equipped with two morphisms $r\: C \to P\x[Y] Z$ and $l\: C\to V_i\x[Y] Z$ such that
\sq{C}{P\x[Y]Z}{V_i\x[Y]Z}{Z.}{r}{l}{f\st \pi_P}{f\st g_i}
We need to prove that there exists a unique morphism $t\: C\to G\x (V_i\x[Y] Z)$ such that
\pbsqunivv{C}{r}{l}{t}{G\x (V_i\x[Y]Z)}{P\x[Y] Z}{V_i\x[Y] Z}{Z.}{h_i} {\pr{2}}{f\st \pi_P}{f\st g_i}
We define 
$$t:= (\pr{1}\c\alpha_i \c q) \x l,$$
where $q\: C\to P\x[Y] V_i$ is the morphism given by the universal property of $P\x[Y] V_i$, as shown in the following diagram:
\pbsqunivv{C}{g_i\st f \c l}{\pi_P \st f\c r}{q}{P\x[Y] V_i}{V_i}{P}{Y.}{\pi_P\st g_i}{g_i\st \pi_P}{\pi_P}{g_i}
It is clear that $\pr{2}\c t=l$ and so it remains to prove that $h_i\c t=r$ and to do so we use the universal property of the pullback $P\x[Y] Z$. Looking at the post-composition with $f\st \pi_P$, we have the following commutative diagram:
\begin{cd}
	{C} \arrow[rr,"r"] \arrow[dd,"t"'] \arrow[rd,"l"]\& \& {P\x[Y] Z} \arrow[dd,"f\st \pi_P"]\\
	\& {V_i\x[Y]Z} \arrow[rd,"f\st g_i"]\& \\
	{G\x(V_i\x[Y]Z)} \arrow[ru,"\pr{2}"] \arrow[d,"h_i"']\& \& {Z}\\
	{P\x[Y]Z.} \arrow[rru,"f\st \pi_P"']\& \& 
\end{cd}
On the other hand, the post-composition with $\pi_P \st f$ gives the following commutative diagram:
\begin{cd}
	{C} \arrow[rrr,"r"] \arrow[d,"t"'] \arrow[rrd,"q"] \& \& \&{P\x[Y]Z} \arrow[ddd,"\pi_P \st f"] \\
	{G\x(V_i\x[Y]Z)} \arrow[dd,"h_i"'] \arrow[rd,"\id{G}\x g_i \st f"]\& \& {P\x[Y]V_i} \arrow[rdd,"\pi_{P}\st g_i"] \& \\
	\& {G\x V_i} \arrow[ru,aiso,"\alpha_i^{-1}"']\& \& \\
	{P\x[Y] Z} \arrow[rrr,"\pi_{P} \st f"'] \& \& \& {P,}
\end{cd}
where every internal diagram is commutative by definition except the upper-left one which is commutative because $G\x V_i$ has the universal property of the product and we have the following commutative diagrams
\begin{cd}
	{C} \arrow[r,"q"] \arrow[dd,"t"]\& {P\x[Y]V_i} \arrow[r,"\alpha_i"]\& {G\x V_i} \arrow[dd,"\pr{1}"] \\
	\& {G\x(V_i\x[Y] Z)} \arrow[rd,"\pr{1}"] \& \\
	{G\x(V_i\x[Y] Z)} \arrow[r,"\id{G}\x g_i \st f"] \arrow[ru,equal,""] \& {G\x V_i} \arrow[r,"\pr{1}"]\& {G} 
	\end{cd}
\vspace{1.5mm}
\begin{cd}
	{C} \arrow[rd,"l"] \arrow[r,"q"] \arrow[dd,"t"]\& {P\x[Y]V_i} \arrow[rdd,bend left=15,"\pi_{P}\st g_i"] \arrow[r,"\alpha_i"]\& {G\x V_i} \arrow[dd,"\pr{2}"] \\
	\& {V_i\x[Y]Z} \arrow[rd," g_i \st f"]\& \\
	{G\x(V_i\x[Y] Z)}  \arrow[r,"\id{G}\x g_i \st f"] \arrow[ru,"\pr{2}"] \& {G\x V_i} \arrow[r,"\pr{2}"]\& {V_i.} 
\end{cd}
Hence, we have shown that $G\x (V_i\x{Y} Z)$ satisfies the universal property of the pullback $(P\x[Y]Z)\x[Z] (V_i\x [Y] Z)$ and so there exists an isomorphism
$$\phi_i\: (P\x[Y]Z)\x[Z] (V_i\x [Y] Z) \to G\x (V_i\x[Y] Z).$$
We now have to prove that $\phi_i$ is $G$-equivariant, i.e.\ that the following diagram is commutative
\sq{G\x((P\x[Y]Z)\x[Z] (V_i\x [Y] Z))}{(P\x[Y]Z)\x[Z] (V_i\x [Y] Z)}{G\x (G\x (V_i\x[Y] Z))}{G\x (V_i\x[Y] Z),}{\psi_i}{\id{G}\x \phi_i}{\phi_i}{\theta_i}
where $\psi_i\: G\x((P\x[Y]Z)\x[Z] (V_i\x [Y] Z)) \to (P\x[Y]Z)\x[Z] (V_i\x [Y] Z)$ is the action defined using Construction~\ref{pullact}. We want to use the universal property of the pullback satisfied by $G\x(V_i\x[Y] Z)$ and so we start considering the post-composition with $h_i$ and we obtain the following  diagram:
\begin{cd}
	{G\x((P\x[Y]Z)\x[Z] (V_i\x [Y] Z)) } \arrow[r,"\psi_i"] \arrow[dd,"\id{G}\x \phi_i"']  \arrow[rd,bend right=10,"\id{G}\x (f\st \pi_{P})\st(f\st g_i)"]\&{(P\x[Y]Z)\x[Z] (V_i\x [Y] Z)} \arrow[r,"\phi_i"] \arrow[rdd,bend left=10,"(f\st \pi_{P})\st(f\st g_i)"] \& {G\x (V_i\x[Y] Z)} \arrow[dd,"h_i"]\\
	\& {G\x (P\x[Y] Z)} \arrow[rd,"\psi"]\& \\
    {G\x (G\x (V_i \x[Y] Z))} \arrow[r,"\theta_i"']	\& {G\x (V_i \x[Y] Z)} \arrow[r,"h_i"] \& {P\x[Y] Z.}
\end{cd}
    We notice that the central and the right diagrams commute by the definitions of $\psi_i$ and $\phi_i$ and so it suffices to prove that the left diagram commutes. To do so we can use the universal property of $P\x[Y] Z$. The following commutative diagrams show that both the post-composition with $f\st \pi_P$ and $\pi_{P}\st f$ work:
    \begin{cds}{7}{2}
    	{G\x((P\x[Y]Z)\x[Z] (V_i\x [Y] Z))} \arrow[rr,"\id{G}\x((f\st \pi_P)\st(f\st g))"] \arrow[dd,"\id{G}\x \phi_i"']\&  \&{G\x(P\x[Y]Z)} \arrow[r,"\psi"] \arrow[d,"\id{G}\x f\st \pi_P"] \arrow[dd,bend right=50,"\pr{2}"']\& {P\x[Y] Z} \arrow[dd,"f\st \pi_{P}"]\\
    	\& \& {G\x Z} \arrow[rd,"\pr{2}"]\& \\
    	{G\x( G\x (V_i\x[Y] Z))} \arrow[r,"\theta_i"'] \arrow[rruu,"\id{G}\x h_i"]\& {G\x(V_i\x[Y]Z)} \arrow[r,"h_i"']\& {P\x[Y] Z} \arrow[r,"f\st \pi_{P}"']\& {Z}
    \end{cds}
    \begin{cds}{7}{2}
    \hspace{-2.8ex}	{G\x((P\x[Y]Z)\x[Z] (V_i\x [Y] Z))} \arrow[rr,"\id{G}\x((f\st \pi_P)\st(f\st g))"] \arrow[dddd,"\id{G}\x \phi_i"'] \&[-2.7ex]  \&{G\x(P\x[Y]Z)} \arrow[r,"\psi"] \arrow[d,"\id{G}\x \pi_{P}\st f"] \& {P\x[Y] Z} \arrow[dddd,"\pi_{P}\st f"]\\
    	\& \& {G\x P} \arrow[rddd,bend left=15,"\rho"]\& \\
    	\& {G\x (P\x[Y] V_i)} \arrow[ru,"\id{G}\x f\st g_i"] \arrow[r,"\beta_i"]\& {P\x[Y] V_i} \arrow[rdd,"f\st g_i"]\& \\
    	\& {G\x (G\x V_i)} \arrow[u,"\id{G}\x \alpha_i^{-1}"] \arrow[r,"\gamma_i"]\& {G\x V_i} \arrow[u,"\alpha_i^{-1}"']\& \\
    	{G\x( G\x (V_i\x[Y] Z))} \arrow[ru,"\id{G}\x(\id{G}\x g_i\st f)"'] \arrow[rruuuu, bend left=20,"\id{G}\x h_i"]\arrow[r,"\theta_i"'] \& {G\x(V_i\x[Y]Z)}  \arrow[r,"h_i"']\arrow[ru,"\id{G}\x g_i\st f"']\& {P\x[Y] Z} \arrow[r,"\pi_{P}\st f"']\& {P,}
    \end{cds}
where $$\beta_i\: G\x(P\x[Y]V_i) \to P\x[Y] V_i$$ is the action given by Construction~\ref{pullact} and $$\gamma_i: G\x(G\x V_i) \to G\x V_i$$ is the action given by the internal product of $G$ on the product $G\x V_i$.

Now we consider the post-composition with $\pr{2}$ and we obtain the following commutative diagram
\begin{cd}
	{G\x((P\x[Y]Z)\x[Z] (V_i\x [Y] Z))} \arrow[r,"\psi_i"] \arrow[dd,"\id{G}\x \phi_i"'] \arrow[rd,"\id{G}\x((f\st g_i)\st(f\st \pi_{P}))"] \& {(P\x[Y]Z)\x[Z] (V_i\x [Y] Z)} \arrow[r,"\phi_i"] \arrow[rdd, "(f\st g_i)\st(f\st \pi_{P})"]\& {G\x (V_i\x [Y] Z)} \arrow[dd,"\pr{2}"]\\
	\&{G\x (V_i\x [Y] Z)}  \arrow[rd,"\pr{2}"]\& \\
	{G\x (G\x (V_i \x[Y] Z)} \arrow[ru,"\pr{2}"] \arrow[r,"\theta_i"]\&{G\x (V_i \x[Y] Z} \arrow[r,"\pr{2}"] \&{V_i\x[Y] Z.}  
\end{cd}
This shows that $\phi_i$ is $G$-equivariant and so we have proved that $\pi_{P\x[Y]Z}\: P\x[Y] Z \to Z$ is a principal $G$-bundle over $Z$. 
\end{proof}

We are now ready to prove that the quotient prestack $\qst{X}{G}$ is a pseudofunctor (see for instance \cite{Neumann} for a definition of pseudofunctor).

\begin{prop} \label{welldef}
	The quotient prestack $\qst{X}{G}$ of Definition~\ref{qp} is a pseudofunctor.
\end{prop}

\begin{proof}
	The fact that $\qst{X}{G}$ is well-defined on objects follows from the previous Proposition and straightforward observations.
	
	Let $f\: Z\to Y$ be a morphism in $\C$. We have to show that the assignment
	$$\qst{X}{G}(f)\: \qst{X}{G}(Y) \to \qst{X}{G}(Z)$$
	 of Definition~\ref{qp} is a functor. 
	 
	 Let $(P,\alpha)$ be an object of $\qst{X}{G}(Y)$.
	By Proposition~\ref{pullbun}, we have that $P\x[Y] Z$ is a principal $G$-bundle over $Z$ with morphism $\pi_{P}\x[Y] Z:= f\st \pi_{P}$. Moreover, the morphism $\alpha\c \pi_{P}\st f\: P\x[Y] Z\to X$ is $G$-equivariant since the following diagram is commutative
		\begin{cd}
			{G\x (P\x[Y]Z)} \arrow[r,"\psi"] \arrow[d,"\id{G}\x \pi_{P}\st f"']\& {P\x[Y] Z} \arrow[d,"\pi_{P}\st f"]\\
			{G\x P}  \arrow[r,"p"] \arrow[d,"\id{G}\x \alpha"']\& {P} \arrow[d,"\alpha"'] \\
			{G\x X} \arrow[r,"x"'] \& {X,}
		\end{cd}
	where $\psi$ is the action of $G$ on the pullback $P\x[Y]Z$ defined as in Construction~\ref{pullact}. Then we have shown that $(P\x[Y]Z, \alpha \c \pi_{P}\st f)$ is an object of $\qst{X}{G}(Y)$ and so $\qst{X}{G}(f)$ is well-defined on objects. To prove that it is well-defined on morphisms as well, we need to show that given a morphism $\phi\: (P, \alpha) \to (Q, \beta)$ in $\qst{X}{G}(Y)$, the morphism $\qst{X}{G}(f)(\phi)$ given by the universal property of pullback as in the following diagram
	\pbsqunivv{P\x[Y]Z}{\phi\c \pi_P\st f}{\pi_P}{\qst{X}{G}(f)(\phi)}{Q\x[Y]Z}{Q}{Z}{Y,}{\pi_Q \st f}{f\st \pi_Q}{\pi_Q}{f}
	is a morphism in $\qst{X}{G}[Z]$. But this is true by definition and it remains to prove just that $\qst{X}{G}(f)$ preserves identities and compositions.
	
	Let $(P, \alpha)$ be an object of $\qst{X}{G}(Y)$. Then the morphism $\qst{X}{G}(f)(\id{(P,\alpha)})$ is the unique morphism such that the following diagram is commutative
	\pbsqunivv{P\x[Y]Z}{\pi_{P}\st f}{f\st \pi_{P}}{\qst{X}{G}(f)(\id{(P,\alpha)})}{P\x[Y] Z}{P}{Z}{Y.}{\pi_{P}\st f}{f\st \pi_{P}}{\pi_{P}}{f}
	But the morphism $\id{(P\x[Y]Z,\id{(P,\alpha)}\c \pi_{P}\st f)}\: P\x[Y] Z \to P\x[Y] Z $ makes the diagram commute too and so $\qst{X}{G}(f)(\id{(P,\alpha)})=\id{(P\x[Y]Z,\id{(P,\alpha)}\c \pi_{P}\st f)}$ and we have shown that $\qst{X}{G}(f)$ preserves identities.
	
	Let now $\phi\: (P,\alpha) \to (Q,\beta)$ and $\psi\: (Q,\beta) \to (R,\delta)$ be morphisms in $\qst{X}{G}(Y)$. The morphism $\qst{X}{G}(f)(\psi \c \phi)$ is the unique morphism such that the following diagram is commutative
	\pbsqunivv{P\x[Y]Z}{(\psi \c\phi)\c \pi_P\st f}{\pi_P}{\qst{X}{G}(f)(\psi \c \phi)}{R\x[Y]Z}{R}{Z}{Y}{\pi_R\st f}{f\st \pi_R}{\pi_R}{f}
	and so to prove the equality $\qst{X}{G}(f)(\psi \c \phi)=\qst{X}{G}(f)(\psi) \c \qst{X}{G}(f)(\phi)$ it suffices to show that the composite $\qst{X}{G}(f)(\psi) \c \qst{X}{G}(f)(\phi)$ makes the diagram commute as well.
	 We notice that the morphism $\qst{X}{G}(f)(\psi) \c \qst{X}{G}(f)(\phi)$ makes the upper triangle of the above diagram commute, since we have the following commutative diagram
	 \begin{cd}
	 	{P\x[Y]Z} \arrow[r,"\pi_{P}\st f"] \arrow[d,"\qst{X}{G}(f)(\phi)"']\& {P} \arrow[r,"\phi"] \arrow[d,"\phi"]\& {Q} \arrow[r,"\psi"]\& {R} \arrow[dd,equal,""]\\
	 	{Q\x[Y] Z} \arrow[r,"\pi_{Q} \st f"] \arrow[d,"\qst{X}{G}(f)(\psi)"']\& {Q} \arrow[rrd,"\psi"]\&\& \\
	 	{R\x[Y] Z} \arrow[rrr,"\pi_{R}\st f"]\& \& \& {R.}
	 \end{cd}
 Finally, the morphism $\qst{X}{G}(f)(\psi) \c \qst{X}{G}(f)(\phi)$ makes the lower triangle of the diagram commute because the following diagram is commutative
 \begin{cds}{9}{11}
 	{P\x[Y]Z} \arrow[r,"\qst{X}{G}(f)(\phi)"] \arrow[rrd,"f\st \pi_{P}"']\& {Q\x[Y] Z} \arrow[r,"\qst{X}{G}(f)(\psi)"] \arrow[rd,"f\st \pi_{Q}"]\& {R\x[Y]Z} \arrow[d,"f\st \pi_{R}"]\\
 	\& \& {Z}
 \end{cds}
	and so we have shown that $\qst{X}{G}(f)$ preserves compositions. This concludes the proof that $\qst{X}{G}(f)$ is a functor.
To prove that $\qst{X}{G}$ is a prestack we now need to show that it is pseudofunctorial.
Let $Y$ be an object of $\C \op$. We have to show that there exists an invertible natural transformation
$$\iota_{Y}\: \qst{X}{G} (\id{Y}) \to \id{\qst{X}{G}(Y)}.$$
Given $(P,\alpha)\in \qst{X}{G}(Y)$, we define $(\iota_Y)_{(P,\alpha)}\: (P\x[Y]Y,\alpha\c \pi_P\st \id{Y}
)\to (P,\alpha)$ as the unique isomorphism from $P\x[Y]Y$ to $P$ that exists because they are both pullbacks of the pair $(\pi_P, \id{Y})$. It is then straightforward to check that $\iota_{Y}$is a natural transformation.

Let now $g\: W\to Z$ and $f\: Z\to Y$ be morphisms in $\C$. We have to show that there exists an invertible natural transformation
$$\varepsilon_{f,g}\: \qst{X}{G} (f\c g) \to \qst{X}{G}(g) \c \qst{X}{G}(f).$$
Given $(p,\alpha)\in \qst{X}{G}(Y)$, we define $$(\varepsilon_{f,g})_{(P,\alpha)}\: (P\x[Y]W,\alpha\c \pi_P\st (f\c g)
)\to ((P\x[Y]Z)\x[W]Z,\alpha\c \pi_P\st f \c \pi_{P\x[Y]Z}\st g)$$
as the unique isomorphism from $P\x[Y]Z$ to $(P\x[Y]Z)\x[W]Z$ (both pullbacks of the pair $(\pi_P, f\c g)$). It is straightforward to check that $(\varepsilon_{f,g})_{(P,\alpha)}$ is a morphism in $\qst{X}{G}[W]$ and the fact that $\varepsilon_{f,g}$ is a natural transformation follows, then, from the pseudofunctoriality of the pullback. This proves that $\qst{X}{G}$ is a prestack.
\end{proof}

\begin{proof}
	To prove that $\qst{X}{G}$ is a prestack we need to show that it is a pseudofunctor.
	Let $Y$ be an object of $\C \op$. We have to show that there exists an invertible natural transformation
	$$\iota_{Y}\: \qst{X}{G} (\id{Y}) \to \id{\qst{X}{G}(Y)}.$$
	 Given $(P,\alpha)\in \qst{X}{G}(Y)$, we define $(\iota_Y)_{(P,\alpha)}\: (P\x[Y]Y,\alpha\c \pi_P\st \id{Y}
	)\to (P,\alpha)$ as the unique isomorphism from $P\x[Y]Y$ to $P$ that exists because they are both pullbacks of the pair $(\pi_P, \id{Y})$. It is then straightforward to check that $\iota_{Y}$is a natural transformation.
	
	Let now $g\: W\to Z$ and $f\: Z\to Y$ be morphisms in $\C$. We have to show that there exists an invertible natural transformation
	$$\varepsilon_{f,g}\: \qst{X}{G} (f\c g) \to \qst{X}{G}(g) \c \qst{X}{G}(f).$$
	Given $(p,\alpha)\in \qst{X}{G}(Y)$, we define $$(\varepsilon_{f,g})_{(P,\alpha)}\: (P\x[Y]W,\alpha\c \pi_P\st (f\c g)
	)\to ((P\x[Y]Z)\x[W]Z,\alpha\c \pi_P\st f \c \pi_{P\x[Y]Z}\st g)$$
	as the unique isomorphism from $P\x[Y]Z$ to $(P\x[Y]Z)\x[W]Z$ (both pullbacks of the pair $(\pi_P, f\c g)$). It is straightforward to check that $(\varepsilon_{f,g})_{(P,\alpha)}$ is a morphism in $\qst{X}{G}[W]$ and the fact that $\varepsilon_{f,g}$ is a natural transformation follows, then, from the pseudofunctoriality of the pullback. This proves that $\qst{X}{G}$ is a prestack.
	\end{proof}

Analogously to the topological case, we can consider the particular case of $[T/G]$ where $T$ is the terminal object of $\C$ and so the  morphisms with target $T$ are uniquely determined and the category $[T/G](Y)$ is isomorphic to $\Bun{G}{Y}$. 

\begin{defne} \label{class}
	Let $\C$ be a category with terminal object $T$ and let $G$ be a group object in $\C$. The prestack $[T/G]$ is called \dfn{classifying stack} and will be denoted $\clst{G}$.
\end{defne}

\begin{oss} 
	If we consider the category $\Top$ equipped with the standard topology the terminal object is the one-point space and the classifying stacks over $\Top$ are the usual classifying stacks of topological groups.
	\end{oss}

\section{Generalized quotient stacks}\label{st}
In this section we will prove that, if the category $\C$ is cocomplete such that pullbacks preserve colimits and the topology $\tau$ is subcanonical, the quotient prestacks defined in the previous section are stacks. In order to prove this important result, we will show that the quotient prestacks satisfy the gluing conditions required for a stack with respect to the canonical topology.

Firstly, we recall the definition of \predfn{canonical topology} on a category (for a reference see \cite{Lester1}).
\begin{defne}
	Let $\C$ be a category. The \dfn{canonical topology} $\kappa$ on the category $\C$ is the finest Grothendieck topology on $\C$ such that every representable presheaf of $\C$ is a sheaf on the site $(\C,\kappa)$.
	A Grothendieck topology $\tau$ on $\C$ is said \dfn{subcanonical} if it is contained in $\kappa$, i.e.\ if every representable presheaf of $\C$ is a sheaf on the site $(\C,\tau)$.
\end{defne}

A very useful characterization of the sieves of the canonical topology has been mentioned in \cite{Elephant} and studied in detail by Lester in \cite{Lester1}. Lester has shown that the canonical topology is the collection of all \predfn{universal colim sieves} in the sense of the following definition.

\begin{defne}
	Let $\C$ be a category and let $X$ be an object of $\C$. A sieve $S$ on $X$ is a \dfn{colim sieve} if $X=\opn{colim}_S{\opn{dom}}$, where $\opn{dom}\: \C/X \to \C$ is the domain functor. Moreover, $S$ is a \dfn{universal colim sieve} if for every morphism $f\: Y\to X$ in $\C$ the sieve $f\st S$ on $Y$ is a colim sieve.
\end{defne}

In \cite{Lester1}, Lester also gives a basis for the canonical topology under the hypothesis that the category $\C$ is cocomplete, has all pullbacks and has stable and disjoint coproducts. The result is the following:

\begin{teor}[\cite{Lester1}]
	Let $\C$ be a cocomplete category with stable and disjoint coproducts and all pullbacks. Then the family $\fami{f_i\: Y_i \to X}{i\in I}$ is in $\kappa(X)$ if and only if the morphism 
	$$\underset{i\in I}{\coprod} f_i \: \underset{i\in I}{\coprod} Y_i \to X$$
	is a universal effective epimorphism (i.e. $\underset{i\in I}{\coprod} f_i$ is the coequalizer of its kernel pair and the change of basis of $\underset{i\in I}{\coprod} f_i$ along every morphism of $\C$ has the same property).
	\end{teor}
	
Before stating the main theorem of this section, we recall a simple well-known fact that will be useful in the proof of the theorem:
\begin{lemma} 
	\label{colimnat}
	Let  $F,G\: \D \to \C$ be functors and let $\gamma\: G \Rightarrow F$ be a natural transformation. Consider $C\in \C$. If  $\fami{F(i)\ar{f_i} C}{i\in I}$ is a cocone for $F$ over $C$, then $\fami{G(i) \ar{f_i\c \gamma_i} C}{i\in I}$ is a cocone for $G$ over $C$.
\end{lemma}

\begin{proof}
	Straightforward using the naturality of $\gamma$.
\end{proof}

We are now ready to prove the main theorem of this section. 

\begin{teor}\label{arestacks}
	Let $\C$ be a  cocomplete category with pullbacks and a terminal object and such that pullbacks preserve colimits. Let then $\tau$ be a subcanonical Grothendieck topology on $\C$ and $X,G\in \C$  with $G$ a group object. Then the quotient prestack $[X/G]$ is a stack.
\end{teor}

\begin{oss}
	Notice that every cocomplete locally cartesian closed category with a terminal object satisfies the hypothesis of the theorem. For example, every cocomplete quasitopos in the sense of \cite{Penon} satisfies the hypothesis of the theorem. This includes all cocomplete elementary topoi, but also other interesting cases such as cocomplete Heyting Algebras (see \cite{Heyting}) and the category of diffeological spaces originally introduced by Souriau in \cite{Souriau}.
\end{oss}

\begin{proof}
	We are going to show that $[X/G]$ is a stack when we have the canonical topology $\kappa$ on $\C$ and this will imply that it is a stack every time the considered topology is subcanonical, since the principal bundles with respect to a fixed subcanonical topology are principal bundles in $(\C,k)$ as well. 
	
	In order to show that $[X/G]$ is a stack, we will show that the three gluing conditions hold.
	
	Consider a covering family $\mathcal{U}=\fami{f_i\: U_i \to Y}{i\in I}$ of $Y$ for the canonical topology $\kappa$.
	
	Let $(P,\alpha), (Q, \beta) \in [X/G](Y)$ and for every $i\in I$ let $\phi_i\: \restr{ (P,\alpha)}{U_i} \to \restr{ (Q,\alpha)}{U_i} $ be a morphism of $\C$ such that for every $i,j\in I$ we have $\restr{\phi_i}{U_{ij}} = \restr{\phi_j}{U_{ij}}$ (where $U_{ij}$ is given by the pullback of $f_i$ and $f_j$). We want to show that there exists a morphism $\eta\: (P,\alpha) \to (Q,\beta)$ such that $\restr{\eta}{U_i}=\phi_i$ for every $i\in I$.
	
	Let $K$ the pullback given as
	\pbsq{K}{\underset{i\in I}{\coprod}P\x[Y]U_i}{\underset{i\in I}{\coprod}P\x[Y]U_i}{P}{\gamma_1}{\gamma_2}{\underset{i\in I}{\coprod}\pi_P \st f_i }{\underset{i\in I}{\coprod}\pi_P \st f_i}
	 and consider the morphism 
	$$\delta := \underset{i\in I}{\coprod}(\pi_Q\st f_i \c \phi_i)\: \underset{i\in I}{\coprod} P\x[Y] U_i \to Q.$$
	Moreover, we notice that, since pullbacks preserve colimits in $\C$ and $\underset{i\in I}{\coprod} f_i $ is a universal effective epimorphism, the morphism $\underset{i\in I}{\coprod} \pi_P \st f_i $ is an effective epimorphism and so the following is a coequalizer diagram
	\coequalizer{K}{\underset{i\in I}{\coprod}P\x[Y]U_i}{P.}{\gamma_1}{\gamma_2}{\underset{i\in I}{\coprod} \pi_P \st f_i }
	In order to use the universal property of this coequalizer to produce a morphism $\eta\: P \to Q$ we show that the morphism $\delta\: \underset{i\in I}{\coprod} P\x[Y] U_i \to Q$ previously defined coequalizes $\gamma_1$ and $\gamma_2$.
	Since pullbacks preserve colimits, we have that $K$ is isomorphic to the coproduct $\underset{i\in I}{\coprod}\underset{j\in I}{\coprod}(P\x[Y]U_i)\x[P](P\x[Y]U_j)$ and so, to prove that $\delta$ coequalizes $\gamma_1$ and $\gamma_2$, it suffices to prove that, for every $i\in I$ and every $j\in J$, the following diagram is commutative
	\begin{cd}
		{(P\x[Y]U_i)\x[P](P\x[Y]U_j)} \arrow[r,""] \arrow[d,""]\& {P\x[Y]U_j} \arrow[r,"\phi_j"] \& {Q\x[Y]U_j} \arrow[d,"\pi_Q\st f_j"]\\
		 {P\x[Y]U_i} \arrow[r,"\phi_i"']\& {Q\x[Y]U_i}  \arrow[r,"\pi_Q\st f_i"']\& {Q.}
	\end{cd}
But the commutativity of this diagram follows from the fact that  $\restr{\phi_i}{U_{ij}}=\restr{\phi_j}{U_{ij}}$, using that $(P\x[Y]U_i)\x[P](P\x[Y]U_i)$ is isomorphic to $(P\x[Y]U_i)\x[U_i](U_i\x[Y]U_j)$ and $(Q\x[Y]U_i)\x[P](P\x[Y]U_i)$ is isomorphic to $(Q\x[Y]U_i)\x[U_i](U_i\x[Y]U_j)$. This shows that $\delta$ coequalizes $\gamma_1$ and $\gamma_2$ and so, by the universal property of the coequalizer, there exists a unique morphism $\eta\: P \to Q$  such that
	\begin{cd}
	{K}\arrow[r,shift left=0.8ex,"{\gamma_1}"]\arrow[r,shift right=0.8ex,"{\gamma_2}"']\& {\underset{i\in I}{\coprod}P\x[Y]U_i} \arrow[r,"{\underset{i\in I}{\coprod}}\pi_P \st f_i"] \arrow[rd,"\delta",shorten=-0.8ex, bend right=10]\&{P} \arrow[d,"\eta", dashed] \\
	{} \& {} \& {Q.} 
\end{cd}

Let us now observe a general fact that will be useful several times during the proof.
\begin{oss}
	 Since the category $\C$ is cocomplete also the category $[X/G](Y)$ is cocomplete. In fact, given a diagram in $[X/G](Y)$, its colimit can be constructed by taking the colimit in $\C$ of the domains of the given arrows and equipping it with a morphism into $Y$ and a morphism into $X$ both induced by the universal property of the colimit in $\C$.  
\end{oss}
By the previous remark, we know that $\eta$ is a morphism of $G$-bundles from $(P,\alpha)$ to $(Q,\beta)$. Moreover, the equality $\restr{\eta}{U_i}=\phi_i$ holds since the diagram 
\tr{P\x[Y]U_i}{Q\x[Y]U_i}{U_i}{\phi_i}{f_i\st \pi_P}{f_i\st \pi_Q}
is commutative because $\phi_i$ is a morphism of $G$-bundles over $U_i$ and the diagram
\sq{P\x[Y]U_i}{P}{Q\x[Y]U_i}{Q}{\pi_P\st f_i}{\phi_i}{\eta}{\pi_P\st f_i}
commutes by definition of $\gamma$ and $\eta$. 

\noindent This conclude the proof that the first of the gluing conditions holds.

Let now $(P,\alpha),(Q,\beta)\in [X/G](Y)$ and let $\phi,\psi\: (P,\alpha)\to (Q,\beta)$ be morphisms such that for every $i\in I$ $\restr{\phi}{U_i}=\restr{\psi}{U_i}$. We want to prove that under this assumptions we have $\phi=\psi$.
By definition, the restrictions of $\phi$ and $\psi$ fit in the following commutative diagrams
\begin{eqD*}
\pbsqunivvN{P\x[Y]U_i}{\phi\c \pi_P\st f_i}{\pi_{P\x[Y]U_i}}{\restr{\phi}{U_i}}{Q\x[Y]U_i}{Q}{U_i}{Y}{}{\pi_{Q\x[Y]U_i}}{\pi_Q}{f_i}
\quad\quad
\pbsqunivvN{P\x[Y]U_i}{\psi\c \pi_P\st f_i}{\pi_{P\x[Y]U_i}}{\restr{\psi}{U_i}}{Q\x[Y]U_i}{Q}{U_i}{Y.}{}{\pi_{Q\x[Y]U_i}}{\pi_Q}{f_i}
\end{eqD*}
Since $\restr{\phi}{U_i}=\restr{\psi}{U_i}$ by hypothesis, the commutativity of the previous diagrams implies that for every $i\in I$
$$\phi\c \pi_P\st f_i=\psi\c \pi_P\st f_i.$$
But then we can conclude that $\phi=\psi$ because the morphisms $\fami{\pi_P\st f_i}{i\in I}$ are jointly epimorphic since the morphisms $\fami{f_i}{i\in I}$ are jointly regular epimorphic.

\noindent This shows that the second of the gluing conditions holds.

It remains to prove that every descent datum is effective and this is the trickiest part of the entire proof. To prove this last condition, we are going to use sieves instead of covering families as this makes the proof a bit less complex. 

Given a sieve $S$ on $Y$, a descent datum on $S$ for $[X/G]$ is an assignment for every morphism $Z\ar{f} Y$ in $S$ of an object $(W_f, \alpha_f) \in [X/G](Z)$ and, for every pair of composable morphisms 
$$Z'\ar{g}Z\ar{f}Y$$
with $f\in S$, of an isomorphism 
$${\phi^{f,g}} \: g\st((W_f,\alpha_f)) \aisoo(W_{f\c g}, \alpha_{f\c g})$$ such that, given morphisms
$$Z''\ar{h} Z' \ar{g}  Z\ar{f} Y$$
with $f\in S$, the following diagram is commutative
\begin{cd}
	h\st (g\st ((W_f,\alpha_f)) \arrow[d,"", aisos] \arrow[r,"h\st \phi^{f,g}"] \&   h\st ((W_{f\c g}, \alpha_{f\c g}) \arrow[d,"\phi^{f\c g,h}"] \\
	(g\c h)\st ((W_f,\alpha_f)) \arrow[r,"\phi^{f,g\c h}"]\& {(W_{f\c g\c h}, \alpha_{f\c g\c h}),}
\end{cd}
where the isomorphism on the left is given by the fact that $[X/G]$ is a pseudofunctor (it is the inverse of $\varepsilon_{g,h}$ according to the notation introduced in the proof of Proposition~\ref{prest}). 

This descent datum is effective if there exist an object $(W,\alpha)\in [X/G](Y)$ and, for every morphism $Z\ar{f}Y\in S$, an isomorphism 
$$\psi^{f}\: f\st((W,\alpha)) \aisoo W_f$$
such that, given morphisms
$$Z'\ar{g} Z \ar{f} Y$$
with $f\in S$, the following diagram is commutative
\begin{cd}
	g\st (f\st ((W,\alpha)) \arrow[d,"", aisos] \arrow[r,"g\st \psi^{f}"] \&   g\st ((W_{f}, \alpha_{f}) \arrow[d,"\phi^{f,g}"] \\
	(f\c g)\st ((W,\alpha)) \arrow[r,"\psi^{f\c g}"]\& {(W_{f\c g}, \alpha_{f\c g}).}
\end{cd}

\begin{notazione}
	Abusing notation, we will sometimes identify the morphisms in $[X/G](Z)$ with the corresponding morphisms of $\C$ and so, for example, we will write the morphism $\phi^{f,g}$ as $g\st W_f \ar{\phi^{f,g}} W_{f\c g}$.
\end{notazione}
Firstly, we want to extend the assignment given by the descent datum to a functor. Since every morphism of the sieve $S$ is sent to an object of a different category of bundles according to its domain, we need to see all these bundles as objects of the same category of bundles. To do this, we can postcompone the given bundle with the starting morphism in $S$ to obtain an object of $[X/G](Y)$.

We want, now, to construct a functor
$$\Lambda\: S\to [X/G](Y)$$
and so we need to define the action on morphisms of $S$. Given two objects $Z\ar fY$ and $P\ar{t}Y$ of $S$, a morphism from $f$ to $t$ is simply a morphism in the slice category of $\C$ over $Y$, i.e. a morphism $Z\ar{k}P$ in $\C$ such that the following triangle commutes
\tr{Z}{P}{Y.}{k}{f}{t}
Noticing that $f=t\c k$, we define $\Lambda(k)\: (W_f,\alpha_f) \to (W_t,\alpha_t)$ as the following composite
$$W_f\ar{(\phi^{k,t})^{-1}} k\st W_t \ar{\pi_{W_t}\st k} W_t$$
and we observe that this is a morphism in $[X/G](Y)$ because the following diagram commutes
\begin{cd}
	{W_f} \arrow[r,"(\phi^{k,t})^{-1}"] \arrow[d,"\pi\shd{W_f}"']\& {k\st W_t } \arrow[r,"\pi_{W_t}\st k"] \arrow[d,"k\st \pi\shd{W_t}"]\& {W_t} \arrow[d,"\pi\shd{W_t}"]\\
	{Z} \arrow[r,"\id{Z}",equal] \arrow[rd,"f"']\& {Z} \arrow[r,"k"]\& {P} \arrow[ld,"t"]\\
	\& {Y.} \& 
\end{cd}
We now show that $\Lambda$, as such defined, is a functor. 

Let $Z\ar{f}Y$, $P\ar{t}Y$ and $Q\ar{s}Y$ be objects of $S$ and let $Z\ar{k}P$ and $P\ar{l}Q$ be morphisms between them. The following diagram shows, then, that $\Lambda(l\c k)=\Lambda(l) \c \Lambda (k)$:
\begin{cd}
	{W_f}  \arrow[r,"(\phi^{t,k})^{-1}"] \arrow[dd,"(\phi^{s,l\c k})^{-1}"']\& {k\st W_t} \arrow[d,"(k\st\phi^{s,l})^{-1}"']\arrow[r,"\pi\st_{W_t} k"] \& {W_t} \arrow[d,"(\phi^{s,l})^{-1}"]\\
	 \& {k\st(l\st(W_s))} \arrow[ld,"",aisos] \arrow[r,"l\st (\pi\st_{W_s} k)"']\& {l\st W_s} \arrow[d,"\pi\st_{W_s} l"]\\
	{(l\c k)\st W_s} \arrow[rr,"\pi\st_{W_s}(l\c k)"']\& \&{W_s.}
\end{cd}
Moreover, if we consider $Z\ar{f}Y\in S$ and $\id{Z}$ as morphism from $f$ to itself in $S$, in order to show that $\Lambda(\id{Z})=\id{\Lambda(Z)}$ we need to prove that $\phi^{\id{Z},f}=\pi\st_{W_f}  \id{Z}$.

By the compatibility condition of the descent datum, we have that the following square is commutative
\begin{cd}
	\id{Z}\st (\id{Z}\st W_f) \arrow[d,"", aisos] \arrow[r,"\id{Z}\st \phi^{f,\id{Z}}"] \&   \id{Z}\st W_f \arrow[d,"\phi^{f,\id{Z}}"] \\
	\id{Z}\st W_f \arrow[r,"\phi^{f,\id{Z}}"']\& W_f.
\end{cd}
On the other hand, by definition the morphism $\id{Z}\st \phi^{f,\id{Z}}$ satisfies the following diagram 
\pbsqunivv{\id{Z}\st(\id{Z}\st W_f)}   {\phi^{f,\id{Z}}\c (\id{Z}\st\pi\shd{W_f})\st\id{Z}} {\id{Z}\st \id{Z}\st \pi\shd{W_f}} {\id{Z}\st \phi^{f,\id{Z}}} {\id{Z}\st W_f} {W_f}{Z}{Z.}{}{}{\pi\shd{W_f}}{\id{Z}}
	
Combining the first diagram and the upper triangle of the second one, we obtain $$\phi^{\id{Z},f}=\pi\st_{W_f}  \id{Z}$$  and this concludes the proof that $\Lambda$ is a functor.

We can now define the pair $(W,\alpha)$ that will show that the descent datum is effective. 
We define
$$W\: = \colim \Lambda$$
and for every $f\in S$ we call $\sigma_f\: W_f\to W$ the cocone morphism. Considering the morphism $\alpha_f\: W_f\to X$ for every $f\in S$, we obtain a cocone and so we can use the universal property of the colimit to induce the morphism $\alpha\: W\to X$ as the unique morphism such that, for every $f\in S$, the following triangle commutes
\tr{W_f}{W}{X.}{\sigma_f}{\alpha_f}{\alpha}

After having defined the pair $(W,\alpha)$, our next step in the proof is the definition of an isomorphism
$$\psi^f\: f\st W \aisoo W_f$$
for every $f\in S$. 

First of all, we observe that, under our assumption that pullbacks preserve colimits in $\C$, we have 
$$f\st W=\colim (f\st \c \Lambda)$$
with universal cocone given by $f\st \sigma_t$ for every $t\in S$. This allows us to induce the isomorphism $\psi^f$ for every $f\in S$ using the universal property of the colimit and to do so we need to define a cocone $\fami{f\st W_t\ar{\Sigma_t} W_f}{t\in S}$.

We observe that $W_f$ is the colimit of the $(f\st t)\st W_f$ with $t$ that varies in $S$ with universal cocone given by $$(f\st t)\st W_f\ar{\pi\st_{W_f}(f\st t)} W_f$$ for every $t\in S$. Because of this, we can define $\Sigma_t\: f\st W_t \to W_f $ as 
$$f\st W_t \arr{\theta_t} (f\st t)\st W_f \arr{\pi\st_{W_f}(f\st t)} W_f,$$ where $\theta_t$ is defined as the following composite:
$$f\st W_t \ar{\lambda_{f,t}} (t\st f)\st  (W_t) \ar{\phi^{t,t\st f}} W_{t\c t\st f} \ar{(\phi^{f,f\st t})^{-1}} (f\st t)\st (W_f),$$
where $\lambda_{f,t}$ is the isomorphism given by the universal property of the pullback $(t\st f)\st W_t$ as in the following diagram
\begin{cd}
	{f\st W_t} \arrow[rrd,"(t\c \pi\shd{W_t})\st f", bend left=20] \arrow[rd,"\lambda_{f,t}"',aiso] \arrow[rddd,"f\st (t\c \pi\shd{W_t})"', bend right= 30]\& \& \\
	\& {(t\st f)\st W_t} \PB{rd}\arrow[r,""] \arrow[d,""]\& {W_t}\arrow[d,"\pi\shd{W_t}"]\\
	\& {f\st P} \PB{rd}\arrow[r,"t\st f"] \arrow[d,"f\st t"']\& {P} \arrow[d,"t"]\\
	\& {Z} \arrow[r,"f"]\& {Y.}
\end{cd}
Since we want to apply Lemma~\ref{colimnat}, we need to prove that there exist a functor $$K\: S \to [X/G](Y)$$ and a natural transformation $$\Theta\: \Lambda \c f\st \Rightarrow K$$ whose component associated to $t\in S$ is $\theta_t$ for every $t\in S$.

We can define the desired functor as $$K=(\pi\shd{W_f})\st \c f\st \c \dom$$ and, given the morphism 
\tr{R}{P}{Y}{l}{t\c l}{t}
in $S$, the following diagram shows that $\Theta\: \Lambda \c f\st \Rightarrow K$ of components $\theta_t$ for every $t\in S$ is a natural transformation.

$${\begin{cdsN}{12}{2.5}
  {f\st W_{t\c l}} \arrow[rrr,"f\st (\Lambda(l))"] \arrow[dd,"\lambda_{f,t\c l}"']\& \& \& {f\st W_t} \arrow[dd,"\lambda_{f,t}"]\\
	\&{((t\c l)\st f)\st (l\st W_l)} \arrow[rru,""] \arrow[d,"",aiso] \& \& \\
	{((t\c l)\st f)\st W_{t\c l}} \arrow[ru,"((t\c l)\st f)\st((\phi^{t,l})^{-1})"] \arrow[d,"\phi^{t,(t\c l)\st f}"']\& {(l\c (t\c l)\st f)\st W_t} \arrow[r,"",aiso]\& {((t\c l)\st l)\st (t\st f)\st W_t} \arrow[r,""] \arrow[d,"((t\st f)\st l)\st \phi^{t,t\st f}"]\& {(t\st f)\st W_t} \arrow[d,"\phi^{t,t\st f}"] \\
	{W_{f\st(t\c l)\c f}}\arrow[rr,"(\phi^{t\c t\st f,(t\st f) \st l})^{-1}"] \arrow[ru,"(\phi^{t,l\c (t\c l)\st f})^{-1}"] \arrow[dd,"(\phi^{f,f\st (t\c l)})^{-1}"']\& \& {((t\st f)\st l)\st W_{t\c t\st f}}\arrow[r,""] \arrow[ld,"(((t\st f)\st l)\st (\phi^{f,f\st t})^{-1}"]\& {W_{f\st t\c f}} \arrow[dd,"(\phi^{f,f\st t})^{-1}"]\\
	\& {((t\st f)\st l)\st (f\st t)\st W_f} \arrow[rrd,""]\& \& \\
	{\pi\st_{W_f}(f\st (R))} \arrow[ru,"",aiso] \arrow[rrr,"K(l)"']\& \& \& {\pi\st_{W_f}(f\st (P))}
\end{cdsN}}$$
\begin{oss}
	Most of the inner diagrams used to show that the previous diagram commutes are compatibility diagrams for the descent datum. The other ones are given by the explicit definitions of the morphisms in play.
\end{oss}
 By Lemma~\ref{colimnat}, we obtain that $\fami{f\st W_t\ar{\Sigma_t} W_f}{t\in S}$ is a cocone for $W_f$ and so, using the universal property of the colimit, we can define $\psi^f \: f\st W \to W_f$ as the unique morphism such that the following diagram is commutative
 \tr{f\st W_t}{F\st W}{W_f}{f\st \sigma_t}{\Sigma_t}{\psi^f}
 for every $t\in S$. 
 Moreover, since $\theta_t$ is an isomorphism for every $t\in S$, we have that  $\psi^f$ is an isomorphism.
 
 To conclude our proof it suffices to show that, given morphisms
 $$Z'\ar{g} Z \ar{f} Y$$
 with $f\in S$, the following diagram is commutative
 \begin{cds}{8}{7}
 	g\st (f\st ((W,\alpha)) \arrow[d,"", aisos] \arrow[r,"g\st \psi^{f}"] \&   g\st ((W_{f}, \alpha_{f}) \arrow[d,"\phi^{f,g}"] \\
 	(f\c g)\st ((W,\alpha)) \arrow[r,"\psi^{f\c g}"]\& {(W_{f\c g}, \alpha_{f\c g}).}
 \end{cds}
 
 Since $g\st (f\st( W))$ is a colimit with universal cocone given by
 $$\fami{g(\st (f\st(W_t))\ar{g\st f\st \sigma_t}f\st (f\st (W))}{t\in S},$$ it suffices to prove that the two paths of the diagram are equal if precomposed by $g\st f\st \sigma_t$ for every $t\in S$. In order to prove this, we need to consider some isomorphisms given by the universal property of pullbacks that are defined as in the following diagrams:
 \pagebreak
 $$\scalebox{0.95}{\begin{cdsN}{6}{-0.8}
 	{((f\st t)\st g \c \varepsilon) \st (t\st f) W_t} \arrow[rrrrdd,"", bend left= 19] \arrow[rddd,"", bend right,end anchor=west] \arrow[rdd,"j"',aiso]\&[-5ex] \&[-3ex] \& \&[2.5ex] \\[-1ex]
 	\& \& {g\st ((t\st f) \st W_t)} \arrow[rd,"k^{-1}"', aiso] \arrow[rrd,"", bend left=6] \arrow[rddd,"", bend right=16]\& \& \\
 	\& {\varepsilon \st(((f\st t)\st g)\st ((t\st f)\st W_t))} \PB{rrd} \arrow[rr,"\tilde{\varepsilon}"',crossing over] \arrow[d,""]\& \& {((f\st t)\st g)\st ((t\st f)\st W_t)} \arrow[r,""] \arrow[d,""] \& {(t\st f)\st W_t} \arrow[d,"\pi\shd{W_{f\c f\st t}}"] \\
 	\& {(f\c g)\st P} \arrow[rr,"\varepsilon"',crossing over]\& \& {g\st f\st P} \PB{rd} \arrow[d,""] \arrow[r,"(f\st t)\st g"']\& {f\st P} \arrow[d,"f\st t"]\\
 	\& \& \& {Z'} \arrow[r,"g"'] \& {Z} 
  \end{cdsN}}$$
 
 $$\scalebox{0.95}{\begin{cdsN}{6}{1}
 	{((f\st t)\st g \c\varepsilon) W_{f\c f\st t}} \arrow[rdd,"r"', aiso] \arrow[rrrrdd,"\pi\st_{W_{f\c f\st t}}((f\st t)\st g \c \varepsilon)", bend left=17] \arrow[rddd,"", bend right=30, end anchor=west]\&[-2ex]\& \& \&[3.5ex] \\
 	\& \& {g\st W_{f\c f\st t}} \arrow[rd,"s^{-1}"', aiso] \arrow[rrd,"", bend left=6] \arrow[rddd,"", bend right=25, end anchor=north west]\& \& \\
 	\& {\varepsilon \st (((f\st t)\st g)\st W_{f\c f\st t})} \PB{rrd} \arrow[rr,"\varepsilon '"',crossing over] \arrow[d,""]\& \& {((f\st t)\st g)\st W_{f\c f\st t}} \arrow[r,""] \arrow[d,""] \PB{rd} \& {W_{f\c f\st t}} \arrow[d,"\pi\shd{W_{f\c f\st t}}"] \\
 	\& {(f\c g)\st P} \arrow[rr,"\varepsilon"',crossing over]\& \& {g\st f\st P} \arrow[r,""] \arrow[d,""] \PB{rd} \& {f\st P} \arrow[d,"f\st t"]\\
 	\& \& \& {Z'} \arrow[r,"g"'] \& {Z} 
 \end{cdsN}}$$

 $$\scalebox{0.92}{\begin{cdsN}{6}{-1.3}
 	{((f\st t)\st g\c\varepsilon)\st ((f\st t)\st W_f)} \arrow[rdd,"t"', aiso] \arrow[rrrrdd,"", bend left=17] \arrow[rddd,"", bend right=30, end anchor=west]\&[-4ex] \&[-4ex]\& \&[3.5ex] \\
 	\& \& {g\st ((f\st t)\st W_f)} \arrow[rd,"v^{-1}"', aiso] \arrow[rrd,"", bend left=6] \arrow[rddd,"", bend right=16]\& \& \\
 	\& {\varepsilon \st(((f\st t)\st g)\st ((f\st t)\st W_f))} \arrow[rr,"\varepsilon ''"', crossing over] \arrow[d,""] \PB{rrd}\& \& {((f\st t)\st g)\st ((f\st t)\st W_f)} \arrow[r,""]  \arrow[d,""] \PB{rd}\& {(f\st t)\st W_f} \arrow[d,"(f\st t)\st \pi \shd{W_f}"] \\
 	\& {(f\c g)\st P} \arrow[rr,"\varepsilon"',crossing over]\& \& {g\st f\st P} \arrow[d,""] \PB{rd} \arrow[r,""]\& {f\st P}  \arrow[d,"t\st f"]\\
 	\& \& \& {Z'}  \arrow[r,"g"']\& {Z} 
 	\end{cdsN}}$$
 
 The following diagram, then, shows what we were required to prove: 
 $$\scalebox{0.87}{\normalsize\begin{cdsN}{13}{0.4}
 	{g\st(f\st(W_t))} \arrow[rdddddd,"",aiso,bend right= 8, end anchor=north west] \arrow[rr,"g\st f\st \sigma_t"] \arrow[rd,"g\st \lambda_{f,t}"] \arrow[ddddddd,"g\st f\st \sigma_t"']\& \&[1ex]  {g\st(f\st(W))} \arrow[rr,"g\st \psi^{f}"] \&[1ex]  \& {g\st W_f} \arrow[ddddddd,"\psi^{f,g}"] \\
 	\&  {g\st (t\st f)\st W_t} \arrow[r,"g\st \phi^{t,t\st f}"]\& {g\st W_{f\c f\st t}} \arrow[r,"g\st ((\phi^{f,f\st t})^{-1})"]\& {g\st (f\st t)\st (W_f)} \arrow[ru,"(g\st(\pi\st_{W_f}))\st(f\st t)"] \& \\
 	\&  {((f\st t)\st g)\st((t\st f)\st W_t)} \arrow[u,"k",aisos]\&{((f\st t)\st g)\st W_{f\c f\st t} } \arrow[u,"s",aisos]\& {((f\st t)\st g)\st((f\st t)\st W_f)} \arrow[u,"v", aisos]\&  \\
 	\&  {\varepsilon \st (((f\st t)\st g)\st((t\st f)\st W_t))}\arrow[u,"\tilde{\varepsilon}",aisos] \&{\varepsilon \st (((f\st t)\st g)\st W_{f\c f\st t}) } \arrow[u,"\varepsilon'",aisos]\& {\varepsilon \st (((f\st t)\st g)\st((f\st t)\st W_f))} \arrow[u,"\varepsilon''",aisos]\&  \\
 	\&  {((f\st t)\st g \c \varepsilon)\st ((t\st f)\st W_t)}\arrow[u,"j",aisos] \arrow[r,"((f\st t)\st g\c \varepsilon)\st \phi^{t,t\st f}"{outer sep=1ex}] \arrow[d,"",aiso]\& {((f\st t)\st g \c \varepsilon)\st W_{t\st f \c f}} \arrow[u,"r",aisos] \arrow[r,"((f\st t)\st g \c \varepsilon)\st((\phi^{f,f\st t})^{-1})"{outer sep=1ex}] \arrow[d,"\phi^{f\c f\st t,(f\st t)\st g \c \varepsilon}"]\& {((f\st t)\st g \c \varepsilon)\st ((f\st t)\st W_f)} \arrow[u,"t"] \arrow[d,"",aiso]\&  \\
 	\&  {t\st (((f\c g)\st)W_t)} \arrow[r,"\phi^{t,t\st (f\c g)}"]\& {W_{t\c t\st (f\c g)}} \arrow[r,"(\phi^{f,g\c (f\c g)\st t})^{-1}"] \arrow[d,"(\phi^{f\c g, (f\c g)\st t})^{-1}"] \&{((g\c (f\c g)\st t)\st W_f} \arrow[d,"",aisos]\& \\
 	\&  {(f\c g)\st W_t} \arrow[u,"\lambda_{f\c g,t}"'] \arrow[rd,"(f\c g)\st \sigma_t"]\& {(f\c g)\st (t\st W_{f\c g})} \arrow[rrd,"\pi\st_{W_{f\c g}}((f\c g)\st t)"]\& {(f\c g)\st t)\st (g\st W_f)} \arrow[l,"((f\c g)\st t)\st \phi^{f,g}"'{outer sep=1ex}] \&  \\
 	{g\st (f\st (W))} \arrow[rr,"",aiso]\& \& {(f\c g)\st W} \arrow[rr,"\psi^{f\c g}"']\&  \& {W_{f\c g},} 
 \end{cdsN}}$$
where all the inner diagrams are given either by compatibility conditions of the descent datum or can be proved to be commutative in a direct way using the universal property of the pullbacks involved and the definitions of the various morphisms. 
\end{proof} 

\subsection*{Acknowledgements}
I would like to thank the referee for their useful comments and their advice about the interactions of this paper with the existing literature. Moreover, I would like to thank my supervisor Frank Neumann for his precious advice and his support. Finally, I would like to thank Luca Mesiti for talking to me about the sieves approach to the notion of descent datum.
	\bibliographystyle{abbrv} 
	\bibliography{Bibliography2}
	\end{document}